\documentclass[12pt]{amsart}

\usepackage{amsmath, amsthm, amssymb}
\usepackage{url} % <-- Para p\'aginas web o similar: \url{...}
\usepackage{graphicx}

\usepackage{xcolor}
\usepackage{moreverb}

\usepackage{fancyvrb}

\def\r{\mathbb R}
\def\E{\mathbb E}
\def\t{\mathbf t}
\def\n{\mathbf n}
\def\b{\mathbf b}

\def\e{\mathbf e}
\def\x{\mathbf x}
\def\y{\mathbf y}

\usepackage{listings}
 \lstnewenvironment{code}[1][]
  {\lstset{#1}}% Add/update settings locally
  {}

\newtheorem{theorem}{Theorem}[section]
 \newtheorem{proposition}[theorem]{Proposition}
 
\theoremstyle{definition}

\begin{document}

\markboth{M. E. Aydin, A. Has, B. Yilmaz}
{A non-Newtonian approach in differential geometry of curves: multiplicative rectifying curves}


\title[Multiplicative rectifying curves]{A non-Newtonian approach in differential geometry of curves: multiplicative rectifying curves}

\author{Muhittin Evren Aydin}

\address{Department of Mathematics, Faculty of Science, Firat University, Elazig,  23200 Turkey}
\email{meaydin@firat.edu.tr }

\author{ Aykut Has}

\address{Department of Mathematics, Faculty of Science, Kahramanmaras Sutcu Imam University, Kahramanmaras, 46100, Turkey  }
\email{ahas@ksu.edu.tr }

\author{ Beyhan Yilmaz}

\address{Department of Mathematics, Faculty of Science, Kahramanmaras Sutcu Imam University, Kahramanmaras, 46100, Turkey  }
\email{byilmaz@ksu.edu.tr }

\begin{abstract}
In this paper, we study the rectifying curves in multiplicative Euclidean space of dimension $3$, i.e., those curves for which the position vector always lies in its rectifying plane. Since the definition of rectifying curve is affine and not metric, we are directly able to perform multiplicative differential-geometric concepts to investigate such curves. Having presented several characterizations, we completely classify the multiplicative rectifying curves by means of the multiplicative spherical curves. 
\end{abstract}
\maketitle

%%%%%%%%%%%%%%%%%%%%%%%

\section{Introduction}\label{intro}

%%%%%%%%%%%%%%%%%%%%%%%

\quad Derivative and integral, which play a central role in the infinitesimal (Newtonian) calculus, are the extensions of the aritmetic operations, addition and subtraction. Hence, it is reasonable to expect that alternative arithmetic operations will also engender alternative calculi. In this sense, Volterra and Grossman and Katz, in their pioneering works \cite{grossman,grossman2,volterra}, independently introduced many different calculi (the geometric calculus, the anageometric calculus, and etc.) from the Newtonian calculus where we are interested in the {\it multiplicative calculus}. The multiplicative calculus is so-called because the calculus depends on multiplication and division operations. 

\quad In the last decades, there has been an ascending interest in improving the theory and applications of multiplicative calculus. From a mathematical point of view, the contribution to a non-Newtonian calculus is of own interest. Non-Newtonian approaches can be found, for example, in complex analysis \cite{bashirov2,bashirov4,kadak,uzer}, in differential equations  \cite{bashirov3,riza,waseem,yalcin,yalcin2}, in numerical analysis \cite{aniszewska,boruah,ozyapici,ozyapici2,misirli,yazici}, in algebra \cite{cakmak,cordova}, in variational analysis \cite{torres}, and in spectral and Dirac system theories \cite{goktas,goktas2,gulsen,yilmaz}. In addition, multiplicative calculus has remarkable applications in dynamical systems \cite{aniszewska,aniszewska2,aniszewska3,rybaczuk2}, in economics \cite{cordova2,filip,ozyapici3}, and in image analysis \cite{florack,mora}.

\quad In this paper, we investigate the differential-geometric curves using the tools of multiplicative calculus. As far as the authors are aware, there has been no such attempt in the literature, except the comprehensive book by Svetlin \cite{georgiev} published in 2022. In his book, the author carried out the multiplicative tools for the study of differential-geometric curves, surfaces and higher-dimensional objects. Here, we pursue two goals: The first is to mathematically enrich the multiplicative calculus by adding differential-geometric interpretations. The second is to show that, in some cases, non-Newtonian derivatives and integrals need to be performed in differential geometry instead of the usual derivatives and integrals.

\quad More clearly, consider the following subset of $ \r^2$ (see Figure \ref{fig1})
$$
C=\{(x,y) \in  \r^2: (\log x)^2+(\log y)^2=1, x,y>0 \}.
$$ 
We also can parameterize this set as $x(t)=e^{\cos (\log t)}$ and $y(t)=e^{\sin (\log 
t)}$, $t>0$. If we use the usual arithmetic operations, derivative and integral, then it would not be easy to understand what the set $C$ expresses geometrically. With or without the help of computer programs, we cannot even calculate its basic invariants, e.g., the arc length function $s(t)$ is given by a complicated integral
$$
s(t)=\int^t \frac{1}{u}\left ((\sin(\log u)e^{\cos (\log u)})^2+(\cos(\log u)e^{\sin (\log u)})^2 \right )^{1/2}du.
$$
However, applying the multiplicative tools, we see that $C$ is indeed a multiplicative circle parameterized by the multiplicative arc length whose center is $(1,1)$ and radius $e$, which is one of the simplest multiplicative curves (see Section \ref{sec3}). This is the reason why, in some cases, the multiplicative tools need to be applied instead of the usual ones.

\quad In the case of the dimension $3$, the determination of geometric objects sometimes becomes more difficult when the multiplicative tools are omitted. For example, consider the following parameterized curve (see Fig. \ref{fig1})
\begin{equation}
x(t)=e^{\sec(\log t)/\sqrt{2}}, \quad y(t)=e^{\sec(\log t)\cos(\log t^{\sqrt{2}})/\sqrt{2}}, \quad z(t)=e^{\sec(\log t)\sin(\log t^{\sqrt{2}})/\sqrt{2}}, t>0. \label{intro3}
\end{equation}
As in the case of the multiplicative circle, it would not be easy to work geometrically on this curve by the means of the usual differential geometry. However, from the perspective of multiplicative differential geometry, it is a multiplicative rectifying curve, a type of parameterized space curves introduced by B.-Y. Chen \cite{chen} in 2003. 

\quad On the other hand, if one uses the multiplicative derivative together with the usual arithmetic operations, then one would have some difficulties; for example, it would not satisfy the properties such as linearity, chain rule, Leibniz rule and so on (see \cite{bashirov}). However, these are essential tools to establish a differential-geometric theory. To overcome these difficulties, Svetlin \cite{georgiev} proposed a useful idea, which is explained as follows.

\quad Let $\r^n$ be the real vector space of the dimension $n \geq 2$. A {\it multiplicative Euclidean space} $\E^n_*$ is  the pair $(\r^n, \langle , \rangle_*)$ where $\langle , \rangle_*$ is the so-called {\it multiplicative Euclidean inner product}. We note that the usual vector addition and scalar multiplication on $\r^n$ are now replaced with the multiplicative operations (see Section \ref{sec2}). With these new arithmetic operations, the multiplicative derivative has the properties that a derivative has. Therefore, it is now suitable for our purpose. 

\quad We would like to emphasize the importance of the rectifying curves whose position vector always lies in its own rectifying plane, because of their close relationship to spherical curves, helices, geodesics and centrodes in mechanics, see \cite{chen2,chen3,deshmukh}. In addition, we point out the underlying space of $\E_*^3$ is $\r^3$ and the definition of rectifying curve is affine and not metric. Hence, if we want to examine the geometric features of the rectifying curves, we may directly use the multiplicative differential-geometric concepts. These are the justifications that we consider such curves in $\E^3_*$.

\quad The structure of the paper is as follows. After some preliminaries on multiplicative algebra and calculus (Section \ref{sec2}), we will recall in Section \ref{sec3} the curvatures and Frenet formulas of multiplicative space curves. In Proposition \ref{prop1}, we will characterize multiplicative curves lying on a multiplicative sphere in terms of their curvatures. In Section \ref{sec4}, we will introduce  the multiplicative counterpart to the notion of rectifying curves. We will completely classify multiplicative rectifying curves using multiplicative spherical curves (Theorem \ref{chen.th.1}). Before that, however, we will need to prove some results characterizing rectifying curves in terms of their curvatures and the multiplicative distance function (Propositions \ref{chen.pr.1}, \ref{chen.pr.2}, \ref{chen.pr.3}, \ref{chen.pr.4}).

%%%%%%%%%%%%%%%%%%%%%%%
\section{Preliminaries} \label{sec2}
%%%%%%%%%%%%%%%%%%%%%%%

\quad In the present section we recall the multiplicative arguments from algebra and calculus, cf. \cite{georgiev,georgiev2,georgiev3}. 

\quad Let $\r_*$ be the set of all the positive real numbers. For $a,b \in \r_*$ we set
\begin{eqnarray*}
 a +_* b &=& e^{ \log a + \log b}, \\
a -_* b &=&  e^{ \log a - \log b},  \\
a \cdot_* b &=&  e^{ \log a \cdot \log b},  \\
a /_* b &=&  e^{ \log a / \log b}, \quad b \neq 1.
\end{eqnarray*}
Here, for some $a \in \r_*$, we also have
$$
a^{2_*}=e^{(\log a)^2}, \quad a^{\frac{1}{2}_*}=e^{\sqrt{\log a}}.
$$

\quad It is direct to conclude that the triple $(\r_{*},+_*,\cdot_*)$ is a field. We denote by $\r_*$ this field for convenience. Introduce
\begin{equation*}
 \left \vert a \right \vert_*=\left\{ 
\begin{array}{ll}
a, & a\in [ 1,\infty ) \\ 
1/a, & a \in (0,1).%
\end{array}%
\right.
\end{equation*} 

\quad Given $x \mapsto f(x) \in \r_*$, $x \in I \subset \r_*$, the {\it multiplicative derivative} of $f$ at $x$ is defined by 
$$
f^*(x)=\lim_{h\rightarrow 1}(f(x+_*h)-_*f(x))/_*h.  
$$
In terms of the usual arithmetical operations,
$$
f^*(x)=\lim_{h\rightarrow 1}\left(\frac{f(hx)}{f(x)}\right)^{1/\log h},  
$$
or by L' Hospital's rule, 
\begin{equation*}
f^*(x)=e^{x(\log f(x))'}, 
\end{equation*}
where the prime is the usual derivative with respect to $x$. 

\quad Denote by $f^{*(n)}(x)$ the $n$th-order multiplicative derivative of $f(x)$ which is the multiplicative derivative of $f^{*(n-1)}(x)$, for some positive integer $n > 1$. We call the function $f(x)$ {\it multiplicative differentiable} on $I$ if $f^{*(n)}(x)$ exists where, if necessary, $n$ may be extended in maximum order that will be needed.

\quad We may easily conclude that the multiplicative derivative holds some essential properties as linearity, Leibniz rule and chain rules (see \cite{georgiev2}).

\quad The {\it multiplicative integral} of the function $f(x)$ in an interval $[a,b] \subset \r_*$ is defined by 
\begin{equation*}
  \int_{*a}^{b}f(x) \cdot_* {d_*x}=e^{\int_{a}^{b} \frac{1}{x}\log f(x)dx }.
\end{equation*}

\quad Let $\r_*^n =\{(x_1 ,..., x_n) : x_1 ,..., x_n \in \r_* \}$ and $\x , \y \in \r_*^n$. Then $\r_*^n$ is a vector space on $\r_*$ with the pair of operations
\begin{eqnarray*}
\x +_* \y &=&(x_1 +_* y_1 ,..., x_n +_* y_n)=(x_1  y_1 ,..., x_n  y_n), \\
a\cdot_*\x &=&(a \cdot_* x_1,..., a \cdot_* y_n)=(e^{\log a \log x_1},...,e^{\log a \log x_n} ), \quad a \in \r_*.
\end{eqnarray*}
The elements of {\it multiplicative canonical basis} of $\r_*^n$ are 
$$
 \e_1 = (1_*,0_*...,0_*), \e_2 = (0_*,1_*,...,0_*), ..., \e_n = (0_* ,0_*,..., 1_*),  
$$
where $0_*=1$ and $1_*=e$. Also, denote by $\mathbf{0}_*=(0_*,...,0_*)$ the multiplicative zero vector. 

\quad A positive-definite scalar product on $\r_*^n $ is defined by
$$ \langle \x , \y \rangle _* =x_1 \cdot_* y_1 +_*...+_*x_n \cdot_* y_n= e^{\log x_1 \log y_1+...+\log x_n \log y_n}.$$
We call $ \langle , \rangle _* $ {\it multiplicative Euclidean inner product} and $(\r_*^n , \langle , \rangle _*)$ {\it multiplicative Euclidean space} denoted by $\E_*^n$. Also, we call two vectors $\x$ and $\y$ {\it multiplicative orthogonal} if $ \langle \x , \y \rangle _* =0_*$. The induced {\it multiplicative Euclidean norm} $\| \cdot \|_*$ on $\E_*^n $ is 
$$
 \|\x\|_*=(\langle \x , \y \rangle _*)^{\frac{1}{2}_*} = e^{\sqrt{\log x_1^2+...+ \log x_n^2 }}.
$$ 
A vector $\x$ with $\|\x\|_* =1_*$ is said to be {\it multiplicative unitary}.

\quad Let $\theta \in [0_*,e^{\pi}] $. We then introduce 
$$\cos_* \theta = e^{\cos (\log \theta)}, \quad \arccos_* \phi = e^{\arccos (\log \phi)},$$
for $\phi \in [e^{-1},1_*]$. Then, the {\it multiplicative radian measure} of {\it multiplicative angle} between $\x $ and $ \y$ is defined by
$$ \theta = \arccos_*  \left ( \langle \x , \y \rangle _*/_*  (\|\x\|_* \cdot_* \|\y\|_*) \right ) = \arccos_*  \left ( e^{\frac{\log \langle \x , \y \rangle_*}{\log \|\x\|_* \log \|\y\|_*}}  \right ) . $$

\quad The {\it multiplicative cross product} of $\x$ and $\y$ in $\E_*^3$ is defined by
\begin{equation*}
 \x \times_* \y = (e^{\log x_2\log y_3 -\log x_3\log y_2},e^{\log x_3\log y_1 -\log x_1\log y_3}, e^{\log x_1\log y_2 -\log x_2\log y_1}) .
\end{equation*}
It is direct to prove that the multiplicative cross product holds the standart algebraic and geometric properties. For example, $\x \times_* \y$ is multiplicative orthogonal to $\x$ and $\y$. In addition, $\x \times_* \y = \mathbf{0}_*$ if and only if $\x$ and $\y$ are multiplicative collinear.

\quad A {\it multiplicative line} passing through a point $P=(p_1,p_2,p_3)$ and multiplicative parallel to $\mathbf{v}=(v_1,v_2,v_3)$ is a subset of $\E_*^3$ defined by
$$ \{Q=(q_1,q_2,q_3) \in \E_*^3 : Q=P+_*t\cdot_*\mathbf{v} \} ,$$
where $q_i=e^{\log p_i+ \log t \log v_i}$, $i = 1,2,3$. We point out that the multiplicative parallelism is algebraically equivalent to multiplicative collinearity.
 
\quad A {\it multiplicative plane} passing through a point $P$ and multiplicative orthogonal to $\mathbf{v}$ is a subset of $\E_*^3$ defined by 
$$  \{Q \in \E_*^3 : \langle Q-_*P , \mathbf{v}  \rangle_* =0_*\} ,$$
where 
$$ e^{(\log q_1 - \log p_1)\log v_1+(\log q_2 - \log p_2)\log v_2+(\log q_3 - \log p_3)\log v_3}=0_*.$$

\quad A {\it multiplicative sphere} with radius $r >0_*$ and centered at $C=(c_1,c_2,c_3) \in \E_*^3$ is a subset of $\E_*^3$ defined by 
$$ \{Q \in \E_*^3 : \|Q-_*C\|_*=r \} , $$
where 
$$e^{ (\log q_1 - \log c_1)^2+(\log q_2 - \log c_2)^2+(\log q_3 - \log c_3)^2}=e^{(\log r)^2}.$$

%%%%%%%%%%%%%%%%%%%%%%%%%%%%%%%%
\section{Differential geometry of curves in $\E_*^3$} \label{sec3}
%%%%%%%%%%%%%%%%%%%%%%%%%%%%%%%%

\quad Consider a function  $\x : I \subset \r_{*} \to \E_*^3$ where $s\mapsto \x (s) =(x_1(s),x_2(s),x_3(s))$. Suppose that $x_i (s)$ ($i=1,2,3$) is multiplicative differentiable on $I$, setting $\x^*(s)=(x^*_1(s),x_2^*(s),x_3^*(s)).$

\quad We call the subset $\mathcal{C} \subset \E_*^3$ which is the range of $\x(s)$ a {\it multiplicative curve}. Here $\x(s)$ is said to be a {\it multiplicative parametrization} of $\mathcal{C}$. We also call $\x(s)$ a {\it multiplicative regular parameterized curve} if nowhere $\x^*(s)$ is $\mathbf{0}_*$. If, also, $\| \x^*(s) \|_* =1_*$,  or equivalently,
$$
 e^{(\log x_1^*)^2+(\log x_2^*)^2+(\log x_3^*)^2} =1_*, \quad \text{ for every } s\in I, 
$$ 
then we call that $\mathcal{C}$ is parametrized by {\it multiplicative arc length}. 

\quad We may easily observe that a multiplicative arc length parameter is independent from the multiplicative translations. In addition, one always could find a multiplicative arc length parameter of the curve $ \mathcal{C} $ (see \cite{georgiev}). In the remaining part, unless otherwise specified, we will assume that $\x(s)$ in $\E_*^3$ is a multiplicative arc length parameterized curve. 

\quad We call that $\mathbf{u}(t)=(u_1(t),...,u_n(t))$ is a {\it multiplicative differentiable vector field} along the curve $\mathcal{C}$ if each $u_i(t)$ is multiplicative differentiable on $I$, $i =1,...,n$. If $\mathbf{u}(t)$ and $\mathbf{v}(t)$ are two multiplicative differentiable vector fields on $\mathcal{C}$, then it is direct to conclude
\begin{equation}
(\langle \mathbf{u}(t) , \mathbf{v}(t) \rangle_*)^* = \langle \mathbf{u}^*(t) , \mathbf{v}(t) \rangle_* +_* \langle \mathbf{u}(t) , \mathbf{v}^*(t) \rangle_* . \label{metric}
\end{equation}

\quad Assume that $\x (s) $ is {\it multiplicative biregular}, that is, nowhere $\x^* (s)$ and $ \x^{**}(s)$ are multiplicative collinear. We consider a trihedron $\{\t (s), \n (s), \b (s) \}$ along $\x(s)$, so-called {\it multiplicative Frenet frame}, where 
\begin{equation*}
\left. 
\begin{array}{l}
\t(s)=\x^*(s), \\ 
\n (s)= \x^{**}(s)/_*\| \x^{**}(s)\|_*, \\
\b (s)= \t (s) \times_* \n (s) .
\end{array}%
\right. 
\end{equation*}

 Hence, by setting $\t (s)=(t_1(s),t_2(s),t_3(s))$, $\n (s)=(n_1(s),n_2(s),n_3(s))$ and $\b (s)=(b_1(s),b_2(s),b_3(s))$, we have
\begin{equation*}
\left. 
\begin{array}{l}
t_i=x_i^* ,\\ 
n_i=e^{\frac{\log x_i^{**}}{\sqrt{(\log x_1^{**})^2+(\log x_2^{**})^2+(\log x_3^{**})^2}}}, \\ 
b_i=e^{(-1)^{j+k-1}(\log t_j\log n_k -\log t_k\log n_j)}, i,j,k \in \{1,2,3 \},  
\end{array}%
\right. 
\end{equation*}
where $j \neq  i \neq k$ and $j <k$.

\quad The vector field $\t(s)$ (resp. $\n(s)$ and $\b(s)$) along $\x(s)$ is said to be {\it multiplicative tangent} (resp. {\it principal normal} and {\it binormal}). It is direct to prove that $\{\t (s), \n (s), \b (s) \}$ is mutually multiplicative orthogonal and $\n (s) \times_* \b (s) =\t (s) $ and $\b (s) \times_* \t (s) =\n (s) $. We also point out that the arc length parameter and multiplicative Frenet frame are independent from the choice of multiplicative parametrization \cite{georgiev}.

\quad We define the {\it multiplicative curvature} $\kappa(s)$ and {\it torsion} $\tau(s)$ as
$$
\kappa(s) =\| \x^{**}(s)\|_*= e^{\sqrt{(\log(x_1^{**}(s))^2+(\log(x_2^{**}(s))^2+(\log(x_3^{**}(s))^2}} 
$$
and
$$
\tau(s)=\langle \n^*(s),\b(s) \rangle_* =e^{\log n_1^*(s) \log b_1(s)+\log n_2^*(s) \log b_2(s)+\log n_3^*(s) \log b_3(s)}.
$$
The {\it multiplicative Frenet formulas} are now
\begin{equation*}
\left. 
\begin{array}{l}
\t^* =\kappa \cdot _* \n, \\ 
\n^*=-_*\kappa \cdot_* \t +_* \tau \cdot_* \b \\ 
\b^*=-_* \tau \cdot_* \n, \\
\end{array}
\right.
\end{equation*}
or equivalently,
\begin{equation*}
\left. 
\begin{array}{l}
t_i^* =e^{\log \kappa \log n_i}, \\ 
n_i^*=e^{-\log \kappa \log t_i + \log \tau \log b_i}, \\ 
b_i^*=e^{-\log \tau \log n_i}, \quad  i =1,2,3.\\
\end{array}
\right.
\end{equation*}

\quad We call $\x(s)$ {\it multiplicative twisted} if nowhere $\kappa (s)$ and $\tau(s)$ is $0_*$. The multiplicative analogous of the fundamental theorem for space curves is the following (see \cite[p. 132-135]{georgiev}).
\begin{theorem}[Existence] \cite{georgiev}
Given multiplicative differentiable functions $f(s)>0_*$, $g(s)$ on $I$. Then, there is a unique multiplicative parametrized curve whose $\kappa(s)=f(s)$ and $\tau(s)=g(s)$. 
\end{theorem}

\begin{theorem}[Uniqueness] \cite{georgiev}
Given two multiplicative parametrized curves $\x(s)$ and $\y(s)$, $s\in I$, whose $\kappa_{\x}(s)=\kappa_{\y}(s)$ and $\tau_{\x}(s)=\tau_{\y}(s)$. Then, there is a multiplicative rigid motion $F$ such that $\y(s)=F(\x(s))$.
\end{theorem}

\quad The {\it multiplicative} {\it osculating} (resp. {\it normal} and {\it rectifying}) {\it plane} of $\x(s)$ at some $s \in I$ is a multiplicative plane passing through $\x(s)$ and multiplicative orthogonal to $\b(s)$ (resp. $\t(s)$ and $\n(s)$). 

\quad We notice that $\x(s)$ is a subset of a multiplicative line if and only if $\kappa(s)$ is $0_*$ on $I$. Analogously, $\x(s)$ is a subset of multiplicative osculating plane itself if and only if $\tau(s)$ is $0_*$ on $I$. We call $\x(s)$ is a {\it multiplicative spherical curve} if it is a subset of a multiplicative sphere.

\quad In terms of the multiplicative Frenet frame, we can write the decomposition for $\x(s)$ as
\begin{equation}
\x(s) =  \langle \x (s), \t(s) \rangle _* \cdot_* \t(s) +_*  \langle \x (s), \n(s) \rangle _* \cdot_* \n(s) +_*  \langle \x (s), \b(s) \rangle _* \cdot_* \b(s), \label{decomp-1}
\end{equation}
or equivalently,
\begin{equation}
x_i(s)=e^{ \log \langle \x (s), \t(s) \rangle _* \log t_i + \log \langle \x (s), \n(s) \rangle _* \log n_i + \log \langle \x (s), \b(s) \rangle _* \log b_i }, \quad i =1,2,3. \label{decomp-2}
\end{equation}
Here we call the functions $ \langle \x (s), \t(s) \rangle _*$, $ \langle \x (s), \n(s) \rangle _*$ and $ \langle \x (s), \b(s) \rangle _*$ the {\it multiplicative tangential, principal normal} and {\it binormal components} of $ \x(s)$, respectively. In addition, denote by $\x^{\intercal_*}$ the multiplicative tangential component and by $\x^{\perp_*}$ multiplicative {\it normal} component. Then,
$$
 \| \x^{\intercal_*}(s) \|_* = \langle \x(s), \t(s) \rangle_*=e^{\log x_1 \log t_1 +\log x_2 \log t_2 +\log x_3 \log t_3}, 
$$
and 
\begin{equation*}
\left. 
\begin{array}{l}
\| \x^{\perp_*}(s) \|_* 
=e^{\sqrt{(\log  \langle \x(s), \n(s) \rangle_*)^2+(\log  \langle \x(s), \b(s) \rangle_*)^2}} \\
=e^{\sqrt{(\log x_1 \log n_1 +\log x_2 \log n_2 +\log x_3 \log n_3)^2+(\log x_1 \log b_1 +\log x_2 \log b_2 +\log x_3 \log b_3)^2}}. 
\end{array}
\right.
\end{equation*}

\quad Next, we give a characterization of the multiplicative spherical curves with $\kappa , \tau \neq 0_*$ in terms of their multiplicative position vectors. Without loss of generality, we will assume that $\mathbb{S}_*^2$ is a multiplicative sphere with radius $r>0_*$ and centered at $\mathbf{0}_*$.

\begin{proposition} \label{prop1}
Let $\x(s) \subset \E_*^3$ be a multiplicative twisted curve. If $\x(s) \subset \mathbb{S}_*^2$, then
$$
\x(s)=(e^{-1}/_*\kappa(s))\cdot_*\n(s)+_*((e^{-1}/_*\kappa(s))^*/_*\tau(s))\cdot_*\b(s)
$$
or equivalently
$$
x_i(s)=e^{- \log n_i(s) / \log \kappa(s)+(-1/ \log \kappa(s))^*\log b_i(s)/\log \tau(s)}.
$$
\end{proposition}

\begin{proof}
By the assumption, because $\x(s) \subset \mathbb{S}_*^2$, we have $\| \x(s) \|_* =r$, for every $s$. It is equivalent to
\begin{equation}
\langle \x(s), \x(s) \rangle_* =e^{(\log r)^2}. \label{prop1-1}
\end{equation}
If we take multiplicative differentiation of both-hand sides in Eq. \eqref{prop1-1} with respect to $s$ then we have
$$
(\langle \x(s), \x(s) \rangle_* )^*=\left (e^{ (\log r)^2}\right )^*, 
$$
or, by Eq. \eqref{metric}, 
\begin{equation}
e^2\cdot_*\langle \x(s), \x^*(s) \rangle_*=0_*.  \label{prop1-2}
\end{equation}
We understand from Eq. \eqref{prop1-2} that the multiplicative tangential component of $\x(s)$ is 
\begin{equation}
\langle \x(s), \t(s) \rangle_*=0_* . \label{prop1-3}
\end{equation}
Again we take multiplicative differentiation of both-hand sides in Eq. \eqref{prop1-3} with respect to $s$, and we have
$$
(\langle \x(s), \t(s) \rangle_* )^*=(0_*)^*, 
$$
or
$$
\langle \t(s), \t(s) \rangle_* +_* \langle \x(s), \t^*(s) \rangle_*=0_*.
$$
Using the multiplicative Frenet formulas,
$$
1_* +_* \kappa(s) \cdot_* \langle \x(s), \n^*(s) \rangle_*=0_*
$$
or
\begin{equation}
\langle \x(s), \n(s) \rangle_*=e^{-1}/_*\kappa(s). \label{prop1-4}
\end{equation}
If we apply the same arguments in Eq. \eqref{prop1-4} and consider Eq. \eqref{prop1-3}, then we obtain
\begin{equation}
\langle \x(s), \b(s) \rangle_*=(e^{-1}/_*\kappa(s))^*/_*\tau(s). \label{prop1-5}
\end{equation}
The result of the theorem follows by replacing Eqs. \eqref{prop1-3}, \eqref{prop1-4}, \eqref{prop1-5} in Eqs. \eqref{decomp-1} and \eqref{decomp-2}, respectively.
\end{proof}

\quad As a direct conseqeunce of Proposition \ref{prop1}, if $\x(s)$ is multiplicative twisted and if $\x(s) \subset \mathbb{S}_*^2$, then
$$
r=\left ( (e^{-1}/_*\kappa(s))^{2_*}+_* ((e^{-1}/_*\kappa(s))^*/_*\tau(s)))^{2_*}) \right )^{\frac{1}{2}_*}
$$
or equivalently,
\begin{equation*}
(\log r)^2= \left ( \frac{1}{\log \kappa(s)} \right)^2 +\left ( \frac{e^{s (-1/ \log \kappa(s))'}}{\log \tau(s)}\right)^2,
\end{equation*}
where the prime $'$ denotes the usual derivative with respect to $s$.

\quad We point out that Proposition \ref{prop1} is valid for a multiplicative twisted curve. However, even in the case $\tau(s)=0_*$, we may use the arguments given in its proof. More explicitly, assume that $\x(s) \subset \mathbb{S}_*^2$ with $\tau(s)=0_*$. Then, if we take multiplicative differentiation in Eq. \eqref{prop1-4}, we derive that $\langle \x(s), \b(s) \rangle_*=0_*$, implying that $\x(s)$ and $\n(s)$ are multiplicative collinear due to Eq. \eqref{prop1-3}. Hence, we conclude that $\x(s)=(e^{-1}/_*\kappa(s))\cdot_*\n(s)$ and $r=|e^{-1}/_*\kappa(s)|_*$, or equivalently,
$$
x_i(s)=e^{(-1/\log \kappa(s)) \log n_i(s)} \quad \text{and} \quad r=|e^{-1/\log \kappa(s)}|_*, \quad i =1,2,3.
$$

\quad As an example, we may take such a multiplicative spherical curve of $\mathbb{S}_*^2$ with radius $1_*$, called {\it multiplicative great circle}, as follows. Let $\Gamma$ pass through $\mathbf{0}_*$ and multiplicative orthogonal to $\mathbf{e}_3$. Consider the multiplicative equator (see Fig. \ref{fig4})
$$
\mathcal{C}=\Gamma \cap \mathbb{S}_*^2 =\{(x,y,z) \in \E_*^3 : (\log x)^2+(\log y)^2=1, z=1 \} .
$$
It can be parametrized by its multiplicative arclength as $\x(s)=\left (e^{\cos \log s}, e^{\sin \log s}, 1 \right )$. Here, by a direct calculation, we may observe that 
$$
\kappa(s)=1_* \quad \text{and} \quad \n(s)= \left (e^{-\cos \log s}, e^{-\sin \log s}, 1 \right ),
$$
implying $\x(s)=e^{-1} \cdot_* \n(s)$.

\begin{figure}[hbtp]
\begin{center}
\includegraphics[width=.4\textwidth]{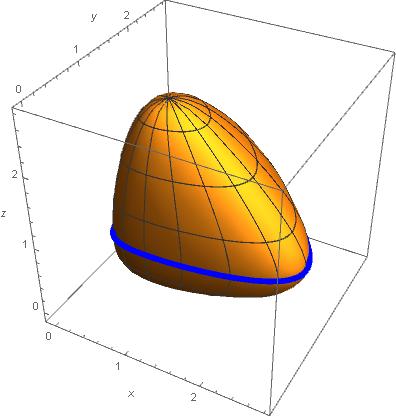} \qquad \includegraphics[width=.4\textwidth]{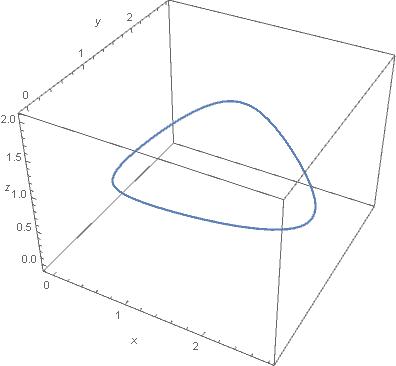}
\end{center}
\caption{A multiplicative equator of $\mathbb{S}_*^2$} \label{fig4}
\end{figure}

%%%%%%%%%%%%%%%%%%%%%%%%%%%%
\section{Multiplicative rectifying curves}\label{sec4}
%%%%%%%%%%%%%%%%%%%%%%%%%%%%

\quad In this section we will introduce the multiplicative analogous of the rectifying curves and then present some characterization and classification results.

\quad Given a multiplicative biregular curve $\x(s)$ in $\E_*^3$, $s \in I \subset \r_*$, where $\{\t (s), \n (s), \b (s) \}$ is the multiplicative Frenet frame. We call $\x(s)$ a {\it multiplicative rectifying curve} if the multiplicative principal normal component is $0_*$. In other words, $\langle \x (s), \n(s) \rangle _* =0_*,$ for every $s \in I$ and so Eq. \eqref{decomp-1} is now
\begin{equation}
\x(s) = \lambda(s) \cdot_* \t (s) +_* \mu (s) \cdot_* \b(s), \label{rect-def}
\end{equation}
where $\lambda(s) =  \langle \x (s), \t(s) \rangle _*$ and $\mu(s) =  \langle \x (s), \b(s) \rangle _*$. In terms of the usual operations, Eq. \eqref{rect-def} writes as
$$
x_i(s)=e^{\log \lambda(s) \log t_i(s) + \log \mu(s) \log b_i(s)}, \quad i =1,2,3 .
$$

\quad The first result of this section characterizes the multiplicative rectifying curves in terms of their multiplicative tangent and binormal components.

\begin{proposition} \label{chen.pr.1}
If $\x(s) \subset \E_*^3$ with $\kappa \neq 0_*$ is a multiplicative rectifying curve, then nowhere the multiplicative torsion is $0_*$ and 
\begin{equation*}
 \langle \x (s), \t(s) \rangle _* =s+_*a =as, \quad  \langle \x (s), \b(s) \rangle _*  =b, \quad a,b \in \r_*, b \neq 0_*. 
\end{equation*}
The converse statement is true as well.
\end{proposition}
\begin{proof}
Assume that $\x(s)$ is a multiplicative rectifying curve. Taking a multiplicative differentiation in Eq. \eqref{rect-def},
$$
\x^*(s)=\lambda^*(s) \cdot_* \t (s)  +_* \lambda(s)  \cdot_* \t^* (s) +_* \mu^*(s)\cdot_* \b(s) +_* \mu(s)\cdot_* \b^*(s). 
$$
The multiplicative Frenet formulas follow
$$
(\lambda^*(s)-_*1_*) \cdot_* \t (s) +_* (\lambda^*(s)\cdot_* \kappa(s) -_*\mu(s) \cdot_* \tau(s)) \cdot_* \n(s) +\mu^*(s)\cdot_* \b(s) =\mathbf{0}_*.
$$
Due to the multiplicative linearly independence,
\begin{equation}
\lambda^*(s) =1_*, \quad  \lambda(s) \cdot_* \kappa (s) = \mu (s)\cdot_* \tau(s), \quad \mu^*(s) = 0_*.  \label{chen3}
\end{equation}
We here conclude that 
$$ 
\lambda (s) =\int_{*}^{s}1_*\cdot_* {d_*u}=e^{\int^{s} \frac{1}{u}\log edu }=e^{\log s + \log a } =s+_*a , \quad a>0,
$$ 
and 
$$ 
\mu (s) =\int_{*}^{s}0_*\cdot_* {d_*u}=e^{\int^{s} \frac{1}{u}\log 1du }=e^{\log b  } =b , \quad b>0.
$$ 
Notice here that $b \neq 0_*$ because otherwise one derives from the middle equation in Eq. \eqref{chen3} that $\kappa(s)$ is $0_*$. This contradicts with biregularity of $\x(s)$. Analogously, $\tau(s)$ is nowhere $0_*$. 

\quad Conversely, suppose that the multiplicative tangential and binormal components are 
$$
\langle \x (s), \t(s) \rangle _* =s+_*a, \quad  \langle \x (s), \b(s) \rangle _* = b,
$$
where $a,b \in \r_*, b \neq 0_*$. Using Eq. \eqref{metric}, 
\begin{eqnarray*}
 0_*&=&(\langle \x (s), \b(s) \rangle _*)^*\\
&=& \langle  \x^* (s), \b(s) \rangle _* +_* \langle \x (s),\b^*(s) \rangle _* \\
&=&\langle  \t (s), \b(s) \rangle _* -_*\tau(s) \langle \x (s),\n(s) \rangle _* \\
&=&-_*\tau(s) \langle \x (s),\n(s) \rangle _*.
\end{eqnarray*}
Since the multiplicative curvatures are nowhere $0_*$, we may easily obtain $\langle \x (s), \n(s) \rangle _* =0_*$. This completes the proof.
\end{proof}

\quad Besides Proposition \ref{chen.pr.1}, several characterizations of multiplicative rectifying curves may be presented as follows.

\begin{proposition} \label{chen.pr.2}
If $\x(s) \subset \E_*^3$ is a multiplicative rectifying curve, then the multiplicative ratio of the multiplicative curvatures is 
$$
\tau(s) /_* \kappa (s) = c\cdot_*s+_*d=e^{c \log s +\log d}, \quad c,d \in \r_*, c \neq 0_*. 
$$
The converse statement is true as well.
\end{proposition}
\begin{proof}
By the middle equation in Eq. \eqref{chen3}, one writes
$$
\tau(s) /_* \kappa (s) = \lambda(s)/_* \mu (s) = (s+_*a)/_*b =e^{\frac {\log s + \log a}{\log b} }, \quad a,b \in \r_*, b \neq 0_*.
$$ 
Setting $\log c =1 / \log b$ and $ \log d = \log a / \log b$ gives the first part of the proof.

\quad Conversely, suppose that $\x (s)$ with $\kappa \neq 0_*$ holds $\tau(s)/_* \kappa (s) = c \cdot_* s+_*d$, for $ c,d \in \r_*, c \neq 0_*. $ Also, we may write
$$
\tau(s) /_* \kappa (s) = (s+_*a)/_*b , \quad a,b \in \r_* , b \neq 0_*
$$

or equivalently,
$$b \cdot_*\tau(s) -_* (s+_*a) \cdot_* \kappa (s) =0_* .$$
We now set
$$ f(s)= \x (s) -_*(s+_*a) \cdot_* \t (s) -_* b \cdot_* \b (s),$$
where $f(s)$ is multiplicative differentiable on its domain. If we take multiplicative differentiation of the last equation and if we use the multiplicative Frenet formulas, then we obtain
$$ f^*(s)=[-_*(s+_*a) \cdot_* \kappa (s) +_* b \cdot_* \tau (s)] \cdot_* \n (s) .$$
Here concludes that $f^*(s)=0_*$ or equivalently $f(s)$ is constant. Then, $\x (s)$ is multiplicative congruent to a multiplicative rectifying curve. 
\end{proof}

\quad We define the {\it multiplicative distance function} of $\x(s)$ as $ \rho (s) = \| \x(s) \|_* $. Hence, we have the following.

\begin{proposition} \label{chen.pr.3}
If $\x(s) \subset \E_*^3$ is a multiplicative rectifying curve with $\kappa >0_*$, then
\begin{equation}
 \rho (s) ^{2_*}=s^{2_*}+_*e^c \cdot_* s+_*e^d, \quad c,d \in \r, d> 0, \label{chen5}
\end{equation}
or equivalently,
$$
e^{(\log \rho (s) )^2} = e^{(\log s )^2+c \log s +d}.
$$
The converse statement is true as well.
\end{proposition}
\begin{proof}
Assume that $\x(s)$ is a multiplicative rectifying curve. It then follows from Eq. \eqref{rect-def} that 
$$
\rho (s)^{2_*}= \langle \x (s), \x (s) \rangle _* = \lambda (s)^{2_*} +_* \mu(s)^{2_*}.
$$ 
Here, by Proposition \ref{chen.pr.1} we know that $\lambda (s)=s+_*a$ and $\mu(s)=b$, for $a,b \in \r_*$, $b \neq 0_*$. Hence
$$\rho (s)^{2_*}=s^{2_*}+_*e^2 \cdot_* a \cdot _* s +_*a^{2_*}+_*b^{2_*}.$$ 
Setting 
$$
e^c=e^2 \cdot_* a=e^{2 \log a}, \quad e^d=a^{2_*}+_*b^{2_*}=e^{(\log a)^2 + (\log b)^2},
$$
we arrive to Eq. \eqref{chen5}. Conversely, let Eq. \eqref{chen5} hold. We will show that $\x(s)$ is a multiplicative rectifying curve. Then,
\begin{equation*}
\langle \x (s), \x (s) \rangle_* = s^{2_*}+_*e^c \cdot_* s+_*e^d.
\end{equation*}
Taking multiplicative differentiation, we have 
$$
e^2 \cdot_*\langle \t (s), \x (s) \rangle _*=e^2 \cdot_*s+_*c,
$$
or, because $e^c/_*e^2=a$,
$$
\langle \t (s), \x (s) \rangle _*=s+_*a.
$$
We again take multiplicative differentiation, obtaining 
$$1_*+_* \kappa (s) \cdot_* \langle \n (s), \x (s) \rangle_*= 1_*,$$ 
where we used the multiplicative Frenet formulas. Hence, $\langle \n (s), \x (s) \rangle_* = 0_*$, completing the proof.
\end{proof}

\begin{proposition} \label{chen.pr.4}
If $\x(s) \subset \E_*^3$ with $\kappa \neq 0_*$ is a multiplicative rectifying curve, then $ \rho(s) $ is nonconstant and  $\| \x^{\perp_*}(s) \|_*$  is constant. The converse statement is true as well.
\end{proposition}

\begin{proof}
Because Propositions \ref{chen.pr.1} and \ref{chen.pr.3}, the only converse statement will be proved. Assume that $ \rho (s) = \| \x (s) \|_* $ is nonconstant and $\| \x^{\perp_*}(s) \|_* = m > 0_*$, $m \in \r_*$. In terms of the multiplicative Frenet frame, the latter assumption yields
$$
m^{2_*}=\langle \x(s), \n(s) \rangle_*^{2_*} +_*\langle \x(s), \b(s) \rangle_*^{2_*},
$$
and so
$$  \langle \x(s), \x(s) \rangle_* = \langle \x(s), \t(s) \rangle_*^{2_*} +_*m^{2_*}.$$
Taking multiplicative derivative and then considering the multiplicative Frenet formulas,
$$  e^2\cdot_* \langle \x(s), \t(s) \rangle_* =e^2\cdot_* \langle \x(s), \t(s) \rangle_* \cdot_*  (1_* +_* \kappa(s) \cdot_* \langle \x(s), \n(s) \rangle_*). $$
Because the assumption, i.e. $\rho(s)$ is not constant, $\langle \x(s), \t(s) \rangle_*$ cannot be $0_*$, implying $\langle \x(s), \n(s) \rangle_* =0_*$. This completes the proof.
\end{proof}

\quad We now introduce 
$$\tan_*s= e^{\tan (\log s)}, \quad s \in (e^{-\pi/2},e^{\pi/2}),$$
and
$$\sec_*s= e^{\sec (\log s)}, \quad s \in [0_*,e^{\pi}], s \neq e^{\pi/2}.$$
It is direct to conclude that $(\sec_* s)^{2*}=1_*+_*(\tan_* s)^{2*}$. 

\quad In what follows, we determine all the multiplicative rectifying curves by means of the multiplicative spherical curves.
 
\begin{theorem} \label{chen.th.1}
Let $\x(s) \subset \E_*^3$ with $\kappa \neq 0_*$ be a multiplicative rectifying curve and $\mathbb{S}_{*}^{2}$ the multiplicative sphere of radius $1_*$. Then, there is a multiplicative reparametrization of $\x(s)$ such that
$$  
\x(\tilde{s}) = (a\cdot_* \sec_* \tilde{s})\cdot_* \y(\tilde{s}) , \quad a \in \r_*, a>0_*,
$$
where $ \y(\tilde{s})$ is a parameterized curve lying in $\mathbb{S}_{*}^{2}$ by multiplicative arc length. The converse statement is true as well.
\end{theorem}

\begin{proof}
Suppose that $0_*$ is included in the domain of $\x$. By Proposition \ref{chen.pr.3}, we have 
$$
\rho (s)^{2_*}  = s^{2_*}+_*e^c \cdot_* s+_*e^d, \quad d > 0. 
$$
Up to a multiplicative translation in $s$, we may take $\rho (s)^{2_*}   = s^{2_*}+_*e^d$. Since $d>0$, there is a constant $a$ such that $e^d=a^{2_*}$. Here $a$ must be greater than $0_*$ because $0_* \in I$. Introduce a curve $\y (s)$ as follows
\begin{equation}
\x(s) =\rho(s) \cdot_* \y(s), \label{chen6}
\end{equation}
where $\rho(s) =(s^{2_*}+_*a^{2_*})^{\frac{1}{2}_*}$. Since $\langle \y(s), \y(s) \rangle_*=1_*$, the curve $\y (s)$ is a subset of $\mathbb{S}_{*}^{2}$. Notice here that 
$$
e^2\cdot_*\langle \y^*(s), \y(s) \rangle_*=0_*.
$$
Taking multiplicative differentiation in Eq. \eqref{chen6},
\begin{equation*}
\x^*(s)  =\rho^*(s)\cdot_* \y(s)+_*\rho(s)\cdot_* \y^*(s),
\end{equation*}
where, because $\rho^*(s)=s/_*\rho(s)$,
\begin{equation*}
\x^*(s)  = ( s/_*\rho(s))\cdot_* \y(s)+_*\rho(s)\cdot_* \y^*(s). 
\end{equation*}
Noting $\langle \x^*(s) , \x^*(s)  \rangle_*=1_*$, we conclude
$$ \langle \y^*(s), \y^*(s) \rangle_*=\left ( 1_* -_* ( s/_*\rho(s))^{2_*} \right) /_* \rho(s)^{2_*}. $$
In terms of the usual operations,
$$ \langle \y^*(s), \y^*(s) \rangle_*=e^{ \left ( \frac{(\log a)^2}{((\log s)^2 + (\log a)^2)^2} \right )}.$$
In order to parametrize $\y(s)$ by multiplicative arc length, we set
$$ \tilde{s} = \int_{*0_*}^{s} \langle \y^*(s), \y^*(s) \rangle_*^{\frac{1}{2}_*} \cdot_* d_*u 
,$$
or
$$ \tilde{s} = \int_{*0_*}^{s}  e^{ \left ( \frac{(\log a)^2}{((\log s)^2 + (\log a)^2)^2} \right )^{\frac{1}{2}}}  \cdot_* d_*u .
$$
By the definiton of the multiplicative integral,
$$ \log \tilde{s} = \left ( \int_{1}^{s} \frac{1}{u} \frac{\log a}{(\log u)^2 + (\log a)^2}  du  \right ). $$
Hence,
$$  \log  \tilde{s}=\arctan \left(\frac{\log s}{\log a} \right), $$
or $\log s = (\log a) \tan (\log  \tilde{s})$. This immediately yields
$$
s=a \cdot_* \tan_*\tilde{s}.
$$
Then, 
$$
\rho(\tilde{s})=\left ((a \cdot_* \tan_*\tilde{s})^{2_*}+_*a^{2_*} \right)^{\frac{1}{2}_*}=\left( (\log a)^2(1+\tan^2 (\log \tilde{s}))\right)^{\frac{1}{2}}=e^{(\log a)\sec(\log \tilde{s})}
$$
or
$$
\rho(\tilde{s})=a \cdot_* \sec_*\tilde{s}.
$$
Considering this into Eq. \eqref{chen6} completes the first part of the proof.

\quad Conversely, suppose that $\tilde{s} \mapsto  \x(\tilde{s})$ is defined by
$$  \x(\tilde{s}) = (a\cdot_* \sec_* \tilde{s})\cdot_* \y(\tilde{s}), $$
where $a >0_*$ and $\y(\tilde{s}) \subset \mathbb{S}_{*}^{2}$ with $ \| \y^*(\tilde{s}) \|_* =1_*$. We will show that $  \x(\tilde{s}) $ is a multiplicative rectifying curve. We first observe $\rho(\tilde{s})=(a\cdot_* \sec_* \tilde{s})$, which is nonconstant. Now we take multiplicative derivative of $\x(\tilde{s})$ with respect to $ \tilde{s}$,
\begin{equation*}
\x^*(\tilde{s}) =  (a\cdot_* \sec_* \tilde{s})^* \cdot_*  \y(\tilde{s}) +_*  (a\cdot_* \sec_* \tilde{s}) \cdot_* \y^*(\tilde{s}) . 
\end{equation*}
Because $(a\cdot_* \sec_* \tilde{s})^*=a\cdot_* \sec_* \tilde{s} \tan_*\tilde{s}$,
\begin{equation}
\x^*(\tilde{s}) =  (a\cdot_* \sec_* \tilde{s}) \cdot_*  ((\tan_* \tilde{s}) \cdot_* \y(\tilde{s}) +_* \y^*(\tilde{s}) ). \label{xy}
\end{equation}
Point out that $\langle \x(\tilde{s}), \y^*(\tilde{s}) \rangle_*=0_*$ because $\langle \y(\tilde{s}), \y^*(\tilde{s}) \rangle_*=0_*$. Multiplying Eq. \eqref{xy} by $\x(\tilde{s})$ in the multiplicative sense, then
$$
\langle \x(\tilde{s}), \x^*(\tilde{s}) \rangle_* =a\cdot_* \sec_* \tilde{s} \cdot_*  \tan_* \tilde{s} \cdot_*   \langle \x(\tilde{s}), \y(\tilde{s}) \rangle_*
$$
or
$$
\langle \x(\tilde{s}), \x^*(\tilde{s}) \rangle_* =(a\cdot_* \sec_* \tilde{s})^{2_*} \cdot_*  \tan_* \tilde{s} .
$$
On the other hand, we may write
$$ \x(\tilde{s}) =\x^{\intercal_*}+_* \x^{\perp *}(\tilde{s})=\langle \x(\tilde{s}), \x^*(\tilde{s})/_*\| \x^*(\tilde{s}) \|_* \rangle_* \cdot_* (\x^*(\tilde{s})/_*\| \x^*(\tilde{s}) \|_* )+_* \x^{\perp *}(\tilde{s}),$$
or
$$ \x(\tilde{s}) =(1/_*\rho(\tilde{s})^{2_*})\langle \x(\tilde{s}), \x^*(\tilde{s}) \rangle_* \cdot_* \x^*(\tilde{s}) +_* \x^{\perp *}(\tilde{s}).$$
Hence,
$$
\x^{\perp_*}(\tilde{s})= \x(\tilde{s}) -_*\tan_* \tilde{s} \cdot_* \x^*(\tilde{s}) 
$$
and
$$
 \langle \x^{\perp *}(\tilde{s}) ,\x^{\perp *}(\tilde{s}) \rangle_* = \langle  \x(\tilde{s}) , \x(\tilde{s}) \rangle_*-_*e^2\cdot_* \langle  \x(\tilde{s}) , \x^*(\tilde{s}) \rangle_*+(\tan_*\tilde{s})^{2_*}\cdot_*\langle  \x^*(\tilde{s}) , \x^*(\tilde{s}) \rangle_*.
$$
By a simple calculation, we find $\langle \x^{\perp *}(\tilde{s}) ,\x^{\perp *}(\tilde{s}) \rangle_* =a^{2_*}$. The remaining part of the proof is by Proposition \ref{chen.pr.4}.

\end{proof}

\quad As an example, consider the following multiplicative spherical curve
$$
\y(s)=\left (e^{\frac{1}{\sqrt {2}}},e^{\frac{1}{\sqrt {2}}\cos (\log s^{\sqrt {2}})},e^{\frac{1}{\sqrt {2}}\sin (\log s^{\sqrt {2}})} \right )
$$
parameterized by multiplicative arc length. Now, by Theorem \ref{chen.th.1} if we set $a=1_*$ and $\x(s)= \sec_* s\cdot_* \y(s)$, then we obtain the rectifying curve given by Eq. \eqref{intro3}.

\begin{figure}[hbtp]
\begin{center}
\includegraphics[width=0.35\textwidth]{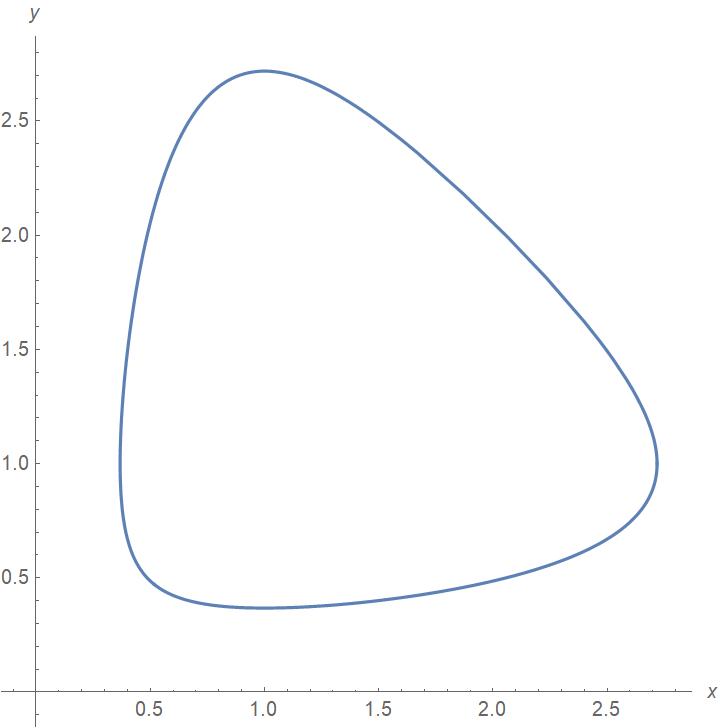} \qquad \includegraphics[width=0.2\textwidth]{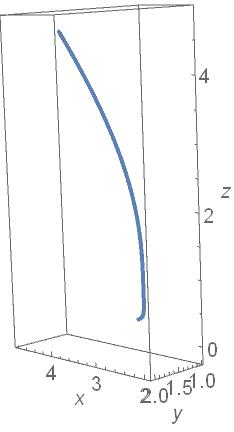}
\end{center}
\caption{Left: a multiplicative circle at centered $(1,1)$ and radius $e$. Right: a multiplicative rectifying curve parametrized by Eq. \eqref{intro3}, $0.5 \leq t \leq 3$.}\label{fig1}
\end{figure}

%\begin{figure}[hbtp]
%\begin{center}
%\includegraphics[width=.4\textwidth]{fig5.jpg}
%\end{center}
%\caption{Multiplicative Chen curves.\label{fig3}}
%\end{figure}

\section*{Conclusions}
\quad Over the last decade, an ascent number of the differential-geometric studies (see \cite{aydin,baleanu,baleanu2,gozutok,yajima,yajima2}) have appeared in which a different calculus (e.g. fractional calculus) from Newtonian calculus is carried out. In the cited papers, the local and non-local fractional derivatives are used. In the case of non-local fractional derivatives, the usual Leibniz and chain rules are known not to be satisfied, which is a major obstacle to establish a differential-geometric theory. In addition, the local fractional derivatives have no remarkable effect on the differential-geometric objects, see \cite{aydin1}. There is a gap in the literature here for a non-Newtonian calculus to be applied these objects.

\quad We highlight two important aspects of our study when a non-Newtonian calculus is performed on a differential-geometric theory: The first is to fill the gap mentioned above. The second is to allow the use of an alternative calculus to the usual Newtonian calculus in differential geometry, the advantages of which have already been addressed in Section \ref{intro}.
 
\section*{Acknowledgments}
This work is supported by The Scientific and Technological Council of Turkey (TUBITAK) with number (123F055).

\section*{Conflict of interest}
 The authors declare that there is no conflict of interest.

%%%%%%%%%%%%%%%%%%%%%%%%%%%%%%%

 \end{document}